\documentclass[10pt,conference]{ita_style}
\usepackage{cite}
\usepackage{amsmath,amssymb,amsfonts}
\usepackage{algorithmic}
\usepackage{graphicx}
\usepackage{textcomp}

\newtheorem{prop}{Proposition}
\DeclareMathOperator*{\argmin}{argmin}
\renewcommand{\baselinestretch}{0.97}
\begin{document}

\title{Matrix Completion for Structured Observations}

\author{\authorblockN{Denali Molitor}
\authorblockA{Department of Mathematics\\
University of California, Los Angeles\\
Los Angeles, CA 90095, USA \\
Email: dmolitor@math.ucla.edu}
\and
\authorblockN{Deanna Needell}
\authorblockA{Department of Mathematics\\
University of California, Los Angeles\\
Los Angeles, CA 90095, USA \\
Email: deanna@math.ucla.edu}}

\maketitle

\begin{abstract}
The need to predict or fill-in missing data, often referred to as matrix completion, is a common challenge in today's data-driven world. Previous strategies typically assume that no structural difference between observed and missing entries exists. Unfortunately, this assumption is woefully unrealistic in many applications. For example, in the classic Netflix challenge, in which one hopes to predict user-movie ratings for unseen films, the fact that the viewer has not watched a given movie may indicate a lack of interest in that movie, thus suggesting a lower rating than otherwise expected. We propose adjusting the standard nuclear norm minimization strategy for matrix completion to account for such structural differences between observed and unobserved entries by regularizing the values of the unobserved entries. We show that the proposed method outperforms nuclear norm minimization in certain settings.
\end{abstract}

\section{Introduction}
Data acquisition and analysis is ubiquitous, but data often contains errors and can be highly incomplete. For example, if data is obtained via user surveys, people may only choose to answer a subset of questions. Ideally, one would not want to eliminate surveys that are only partially complete, as they still contain potentially useful information. For many tasks, such as certain regression or classification tasks, one may require complete or completed data \cite{missingData}. Alternatively, consider the problem of collaborative filtering, made popular by the classic Netflix problem \cite{Netflix, NetflixLessons, RecommenderSys}, in which one aims to predict user ratings for unseen movies based on available user-movie ratings. In this setting, accurate data completion is the goal, as opposed to a data pre-processing task. Viewing users as the rows in a matrix and movies as the columns, we would like to recover unknown entries of the resulting matrix from the subset of known entries. This is the goal in many types of other applications, ranging from systems identification \cite{liu2009interior} to sensor networks \cite{biswas2006semidefinite,schmidt1986multiple,singer2008remark}. This task is known as \textit{matrix completion} \cite{recht2011simpler}.
If the underlying matrix is low-rank and the observed entries are sampled uniformly at random, one can achieve exact recovery with high probability under mild additional assumptions by using nuclear norm minimization (NNM) \cite{candes,fazel,recht,gross,plan}. 

For many applications, however, we expect \textit{structural differences}  between the observed and unobserved entries, which violate these classical assumptions.  By structural differences, we mean that whether an entry is observed or unobserved need not be random or occur by some uniform selection mechanism. Consider again the Netflix problem. Popular, or well-received movies are more likely to have been rated by many users, thus violating the assumption of uniform sampling of observed entries across movies. On the flip side, a missing entry may indicate a user's lack of interest in that particular movie. Similarly, in sensor networks, entries may be missing because of geographic limitations or missing connections; in survey data, incomplete sections may be irrelevant or unimportant to the user.  In these settings, it is then reasonable to expect that missing entries have lower values\footnote{Of course, some applications will tend to have \textit{higher} values in missing entries, in which case our methods can be scaled accordingly.} than observed entries.

In this work, we propose a modification to the traditional NNM for matrix completion that still results in a semi-definite optimization problem, but also encourages lower values among the unobserved entries. We show that this method works better than NNM alone under certain sampling conditions.

\subsection{Nuclear Norm Matrix Completion}
Let $M\in \mathbb{R}^{n_1\times n_2}$ be the unknown matrix we would like to recover and $\Omega$ be the set of indices of the observed entries. Let 
$\mathcal{P}_\Omega : \mathbb{R}^{n_1\times n_2} \to \mathbb{R}^{n_1\times n_2}$, where 
\[[\mathcal{P}_\Omega]_{ij} = \begin{cases} M_{ij} & (i,j)\in \Omega\\ 0 &  (i,j)\not\in \Omega \end{cases}\] 
as in \cite{candes}.
In many applications, it is reasonable to assume that the matrix $M$ is low-rank. 
For example, we expect that relatively few factors contribute to a user's movie preferences as compared to the number of users or number of movies considered. Similarly, for health data, a few underlying features may contribute to many observable signs and symptoms.

The minimization, 
\[\widehat M = \argmin_A \text{rank}(A) \text{ s.t. } \mathcal{P}_\Omega(A) = \mathcal{P}_\Omega(M)\]
recovers the lowest rank matrix that matches the observed entries exactly.
Unfortunately, this minimization problem is NP-hard, so one typically uses the convex relaxation 
\begin{equation}\label{standardNNM}
\widehat M = \argmin_A ||A||_* \text{ s.t. } \mathcal{P}_\Omega(A) = \mathcal{P}_\Omega(M),
\end{equation}
where $||\cdot ||_*$ is the nuclear norm, given by the sum of the singular values, i.e.
$||X||_* := \sum_i \sigma_i(X)$ \cite{candes, fazel, plan, recht}.

\subsection{Matrix Completion for Structured Observations}
We propose adding a regularization term on the unobserved entries to promote adherence to the structural assumption that we expect these entries to be close to 0. We solve 
\begin{equation}\label{min}
\widetilde M = \argmin_A ||A||_* + \alpha ||\mathcal{P}_{\Omega^C}(A)|| \text{ s.t. } \mathcal{P}_\Omega(A) = \mathcal{P}_\Omega(M),
\end{equation}
where $\alpha>0$ and $||\cdot || $ is an appropriate matrix norm. For example, if we expect most of the unobserved entries to be 0, but a few to be potentially large in magnitude, the entrywise $L_1$ norm $||M||_1 = \sum_{ij} |M_{ij}|$ is a reasonable choice.

\subsection{Matrix Completion with Noisy Observations}

In reality, we expect that our data is corrupted by some amount of noise. We assume the matrix $M$, that we would like to recover, satisfies
\[\mathcal{P}_\Omega Y = \mathcal{P}_\Omega M +\mathcal{P}_\Omega Z,\] where $ \mathcal{P}_\Omega Y$ are the observed values, $M$ is low-rank and $\mathcal{P}_\Omega Z$ represents the noise in the observed data.
In \cite{plan}, Cand\'es and Plan suggest using the following minimization to recover the unknown matrix:
\begin{equation}\label{min_noise}
\widehat M = \argmin_A  ||A||_* \text{ s. t. }||\mathcal{P}_\Omega(M-A)||_F< \delta.
\end{equation}
Recall, $||X||_F = \sqrt{\sum_{ij} X_{ij}^2}$.
The formulation above is 
equivalent to
\begin{equation}\label{min_noise_solve}
\widehat M =\argmin_A ||\mathcal{P}_\Omega(M-A)||_F+\rho ||A||_*
\end{equation}
for some $\rho = \rho(\delta)$. The latter minimization problem is generally easier to solve in practice \cite{plan}.

In order to account for the assumption that the unobserved entries are likely to be close to zero, we again propose adding a regularization term on the unobserved entries and aim to solve 
\begin{equation}\label{min_noise_reg}
\widetilde M =\argmin_A ||\mathcal{P}_\Omega(M-A)||_F+\rho ||A||_*+ \alpha ||\mathcal{P}_{\Omega^C}(A)||.
\end{equation}

\section{Numerical Results}

\subsection{Recovery without Noise}
We first investigate the performance of  \eqref{min} when the observed entries are exact, i.e. there is no noise or errors in the observed values.
In Figure~\ref{fig.results}, we consider low-rank matrices $M\in \mathbb{R}^{30\times 30}$. To generate $M$ of rank $r$, we take $M= M_L M_R$, where $M_L \in \mathbb{R}^{30\times r}$ and $M_R\in \mathbb{R}^{r\times 30}$ are sparse matrices (with density 0.3 and 0.5, respectively) and whose nonzero entries are uniformly distributed at random between zero and one. We subsample from the zero and nonzero entries of the data matrix at various rates to generate a matrix with missing entries.
We compare performance of \eqref{min} using $L_1$ regularization on the unobserved entries with standard NNM and report the error ratio $||\widetilde M - M||_F / ||\widehat M - M||_F$ for various sampling rates, where $\widetilde{M}$ and $\widehat{M}$ are the solutions to \eqref{min} and \eqref{standardNNM}, respectively. The regularization parameter $\alpha$ used is selected optimally from the set $\{10^{-1},10^{-2},10^{-3},10^{-4}\}$ (discussed below). 
 Values below one in Figure \ref{fig.results} indicate that the minimization with $L_1$ regularization outperforms standard NNM. Results are averaged over ten trials. As expected, we find that if 
the sampling of the nonzero entries is high, then the modified method \eqref{min} is likely to outperform standard NNM. 

We choose the parameter $\alpha$, for the regularization term, to be optimal among $\alpha \in \{10^{-1},10^{-2},10^{-3},10^{-4}\}$ and report the values used in Figure~\ref{fig.alpha}. For large $\alpha$, the recovered matrix will approach that for which all unobserved entries are predicted to be zero, and as $\alpha$ becomes close to zero, recovery by \eqref{min} approaches that of standard NNM. 

When the sampling rate of the zero entries is low and the sampling of the nonzero entries is high, in addition to \eqref{min} outperforming NNM, we also see that a larger value for $\alpha$ is optimal, supporting the claim that regularization improves performance. Higher $\alpha$ values are also sometimes optimal when the nonzero sampling rate is nearly zero. If there are very few nonzero entries sampled then the low-rank matrix recovered is likely to be very close to the zero matrix. In this setting, we expect that even with standard NNM the unobserved entries are thus likely to be recovered as zeros and so a larger coefficient on the regularization term will not harm performance. When $\alpha$ is close to zero, the difference in performance is minimal, as the regularization will have little effect in this case.

\begin{figure}[htbp]
\centerline{\includegraphics[width = .5\textwidth]{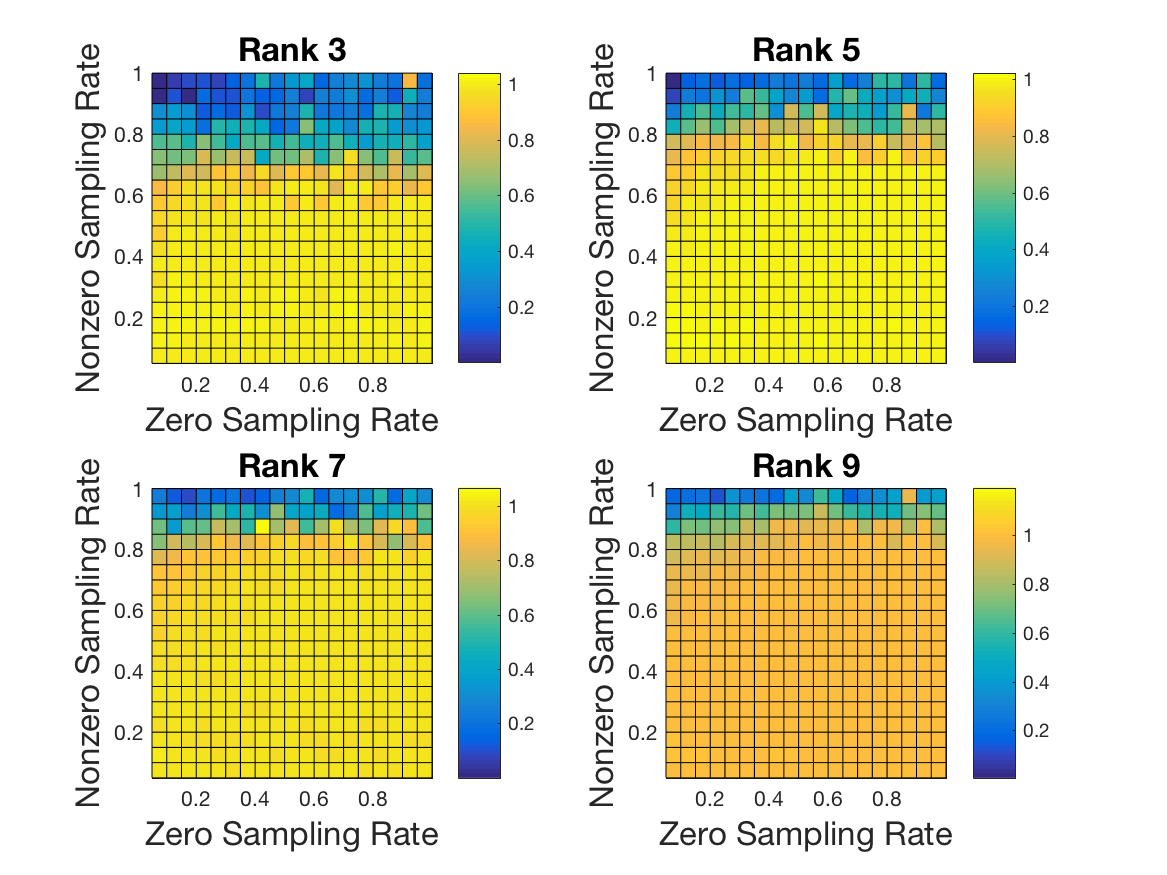}}
\caption{For $\widetilde M$ and $\widehat{M}$ given by \eqref{min} and \eqref{standardNNM}, respectively, with $L_1$ regularization on the recovered values for the unobserved entries, we plot $||\widetilde M - M||_F / ||\widehat M - M||_F$. We consider 30x30 matrices of various ranks and average results over ten trials, with $\alpha$ optimal among $\alpha \in \{10^{-1},10^{-2},10^{-3},10^{-4}\}$. 
}
\label{fig.results}
\end{figure}

\begin{figure}[htbp]
\centerline{\includegraphics[width = .5\textwidth]{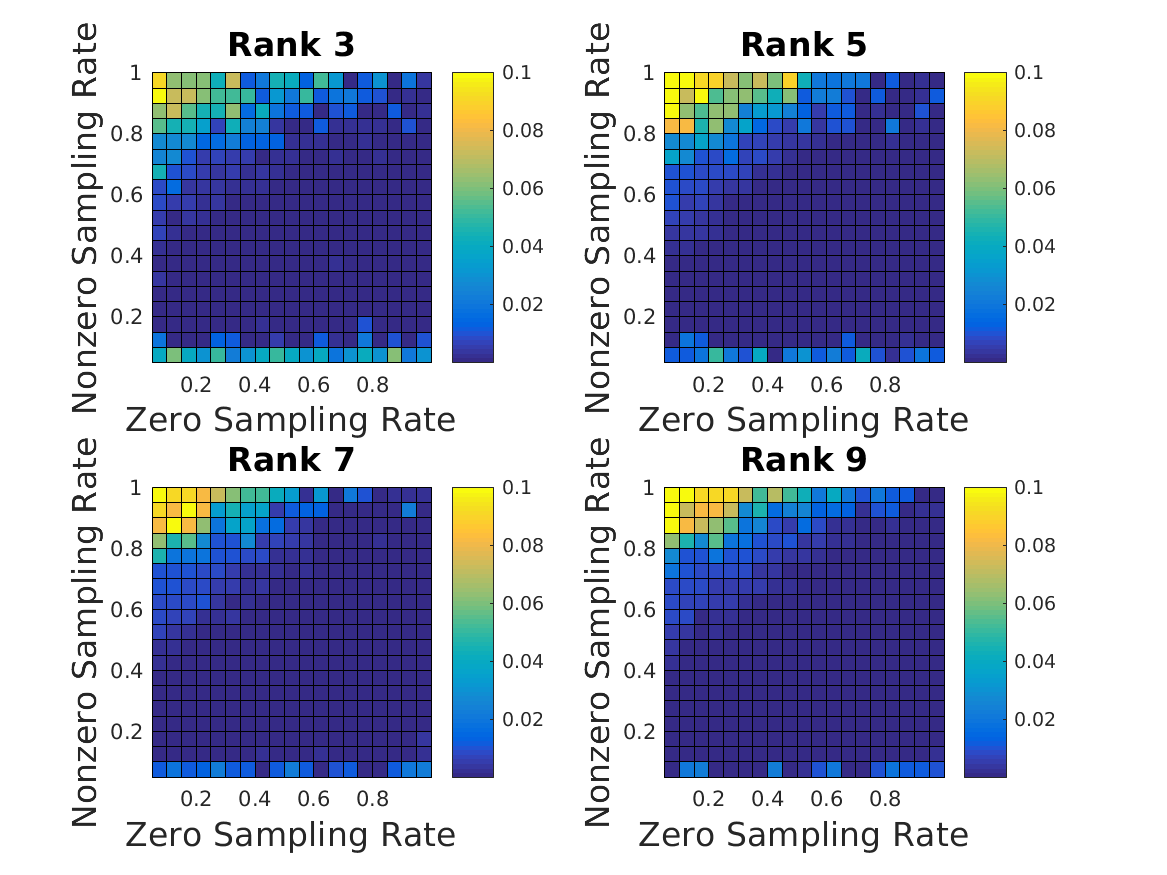}}
\caption{Average optimal $\alpha$ value among $\alpha \in \{10^{-1},10^{-2},10^{-3},10^{-4}\}$ for the minimization given in \eqref{min} with $L_1$ regularization on the recovered values for the unobserved entries. The matrices considered here are the same as in Figure~\ref{fig.results}. }
\label{fig.alpha}
\end{figure}

\subsection{Recovery with Noisy Observed Entries}
We generate matrices as in the previous section and now consider the minimization given in \eqref{min_noise_solve}. Suppose the entries of the noise matrix $Z$ are i.i.d. $N(0,\sigma^2)$. We set the parameter $\rho$, as done in \cite{plan}, to be
\[\rho = (\sqrt{n_{1}}+\sqrt{n_{2}})\sqrt{\frac{|\Omega|}{n_{1}n_{2}}}\sigma.\]
We specifically consider low-rank matrices $M\in \mathbb{R}^{30\times 30}$ generated as in the previous section and a noise matrix $Z$ with i.i.d. entries sampled from $ N(0,0.01)$. Thus we set 
$\rho = 2\sqrt{\frac{{|\Omega|}}{{30}}}\cdot 0.1.$ 
We again report $||\widetilde M - M||_F / ||\widehat M - M||_F$ for various sampling rates of the zero and nonzero entries of $M$ in Figure \ref{fig.resultsNoise}. 
Here, $\widehat M$ and $\widetilde M$ are given by \eqref{min_noise_solve} and \eqref{min_noise_reg} respectively.
We see improved performance with regularization when the sampling rate of the zero entries is low and the sampling of the nonzero entries is high.

\begin{figure}[htbp]
\centerline{\includegraphics[width = .5\textwidth]{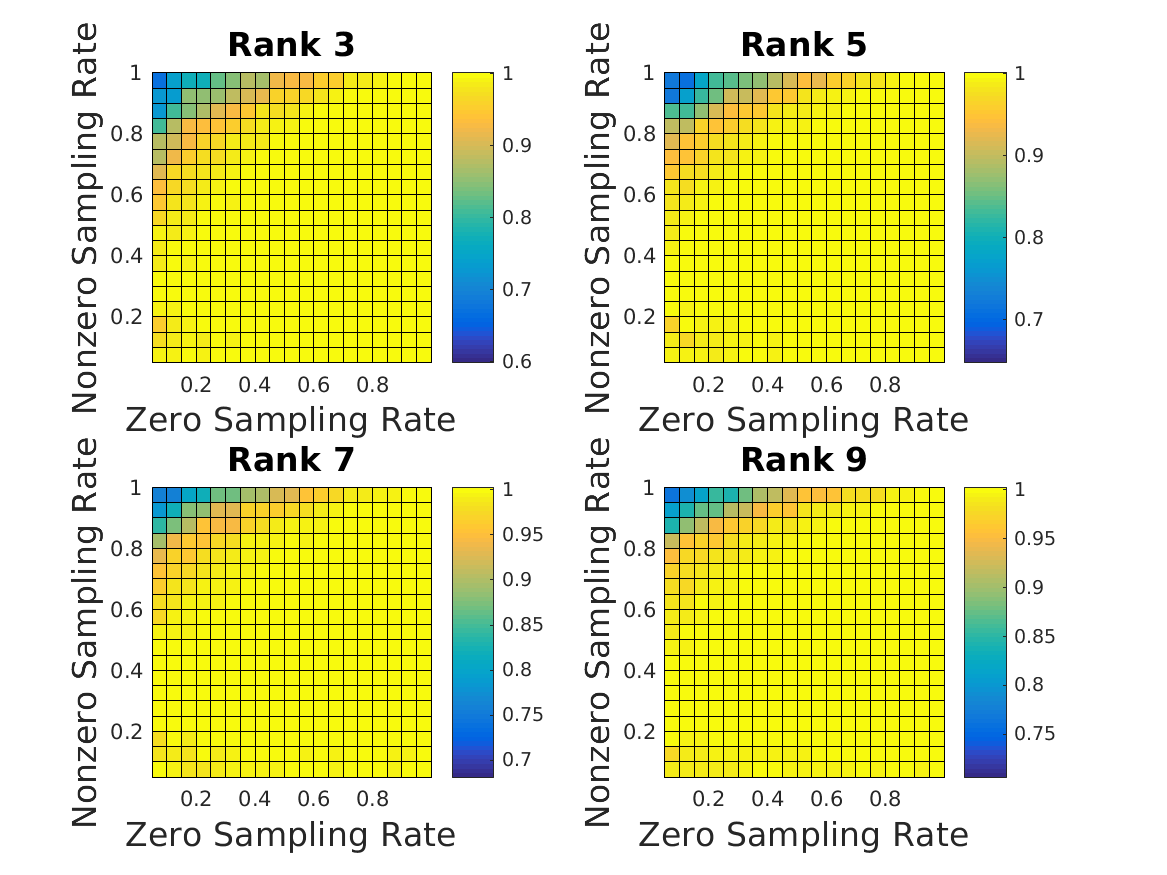}}
\caption{For $\widetilde M$ and $\widehat{M}$ given by \eqref{min} and \eqref{standardNNM}, respectively, with $L_1$ regularization on the recovered values for the unobserved entries, we plot $||\widetilde M - M||_F / ||\widehat M - M||_F$. We consider 30x30 matrices of various ranks with normally distributed i.i.d. noise with standard deviation $\sigma=0.1$ added. We average results over ten trials and with $\alpha$ optimal among $\alpha \in \{10^{-1},10^{-2},10^{-3},10^{-4}\}$. 
}
\label{fig.resultsNoise}
\end{figure}

\subsection{Matrix recovery of health data}
Next, we consider real survey data from 2126 patients responding to 65 particular questions provided by LymeDisease.org. Data used was obtained from the LymeDisease.org patient registry, MyLymeData, Phase 1, June 17, 2017. Question responses are integer values between zero and four and answering all questions was required, that is this subset of the data survey is complete (so we may calculate reconstruction errors). All patients have Lyme disease and survey questions ask about topics such as current and past symptoms, treatments and outcomes. For example, ``I would say
that currently in general my health is: 0-Poor, 1-Fair, 2-Good, 3-Very good, 4-Excellent." 
Although, this part of the data considered is complete, we expect that in general, patients are likely to record responses for particularly noticeable symptoms, while a missing response in a medical survey may indicate a lack of symptoms. Thus, in this setting,
$L_1$ regularization of the unobserved entries is a natural choice. 

Due to computational constraints, for each of the ten trials executed, we randomly sample 50 of these patient surveys to generate a 50x65 matrix. As in the previous experiments, we subsample from the zero and nonzero entries of the data matrix at various rates to generate a matrix with missing entries. We complete this subsampled matrix with both NNM \eqref{standardNNM} and \eqref{min} using $L_1$ regularization on the unobserved entries and report $||\widetilde M - M||_F / ||\widehat M - M||_F$, averaged over ten trials in Figure~\ref{fig.resultsLyme}.
The parameter $\alpha$, for the regularization term, is chosen to be optimal among $\alpha \in \{10^{-1},10^{-2},10^{-3},10^{-4}\}$ and we report the values used in Figure~\ref{fig.alphaLyme}.

The results for the Lyme disease data match closely those found in the synthetic experiments done with and without noise. Regularizing the $L_1$-norm of the unobserved entries improves performance if the sampling of non-zero entries is sufficiently high and sampling of zero entries is sufficiently low.

\begin{figure}[htbp] 
\centerline{\includegraphics[width = .4\textwidth]{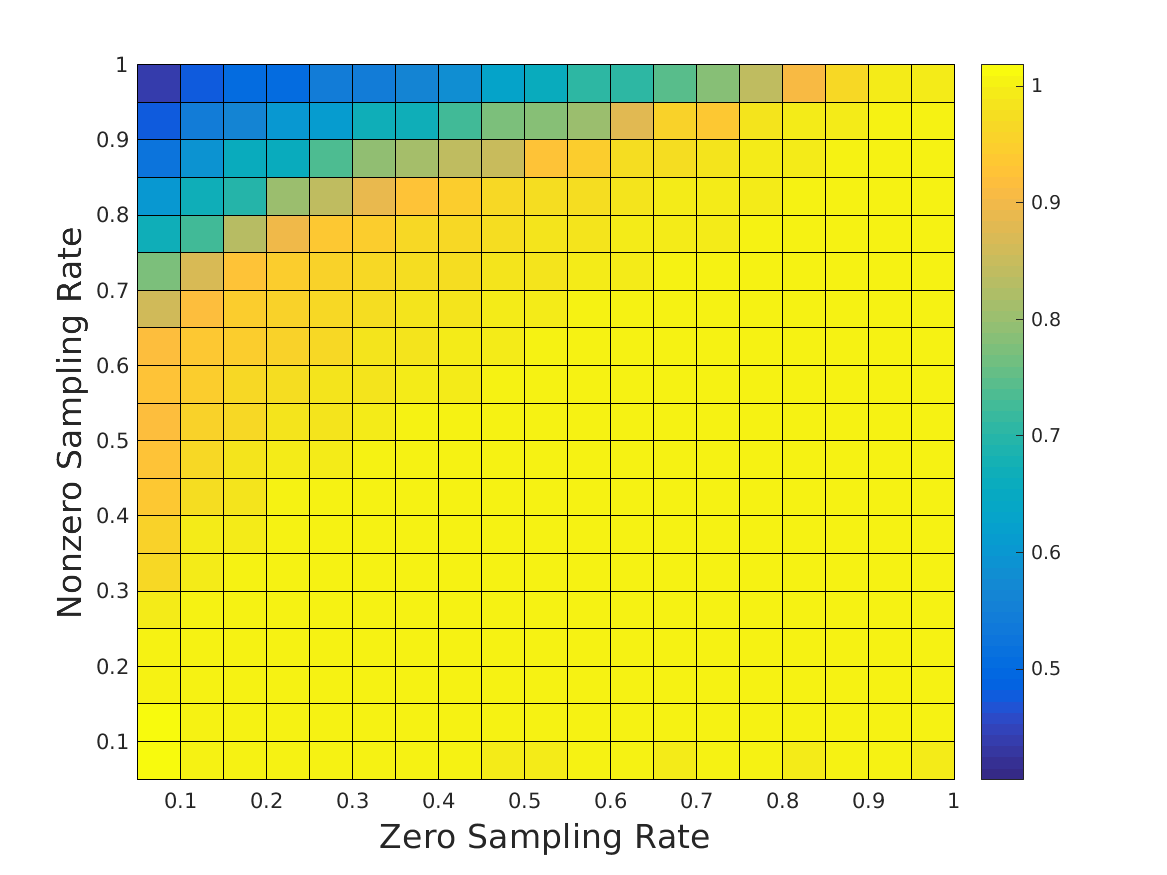}}
\caption{For $\widetilde M$ and $\widehat{M}$ given by \eqref{min} and \eqref{standardNNM}, respectively, with $L_1$ regularization on the recovered values for the unobserved entries, we plot $||\widetilde M - M||_F /||\widehat M - M||_F$. We consider 50 patient surveys with 65 responses each chosen randomly from 2126 patient surveys. We average results over ten trials and with $\alpha$ optimal among $\alpha \in \{10^{-1},10^{-2},10^{-3},10^{-4}\}$. 
}
\label{fig.resultsLyme}
\end{figure}

\begin{figure}[htbp]  
\centerline{\includegraphics[width = .4\textwidth]{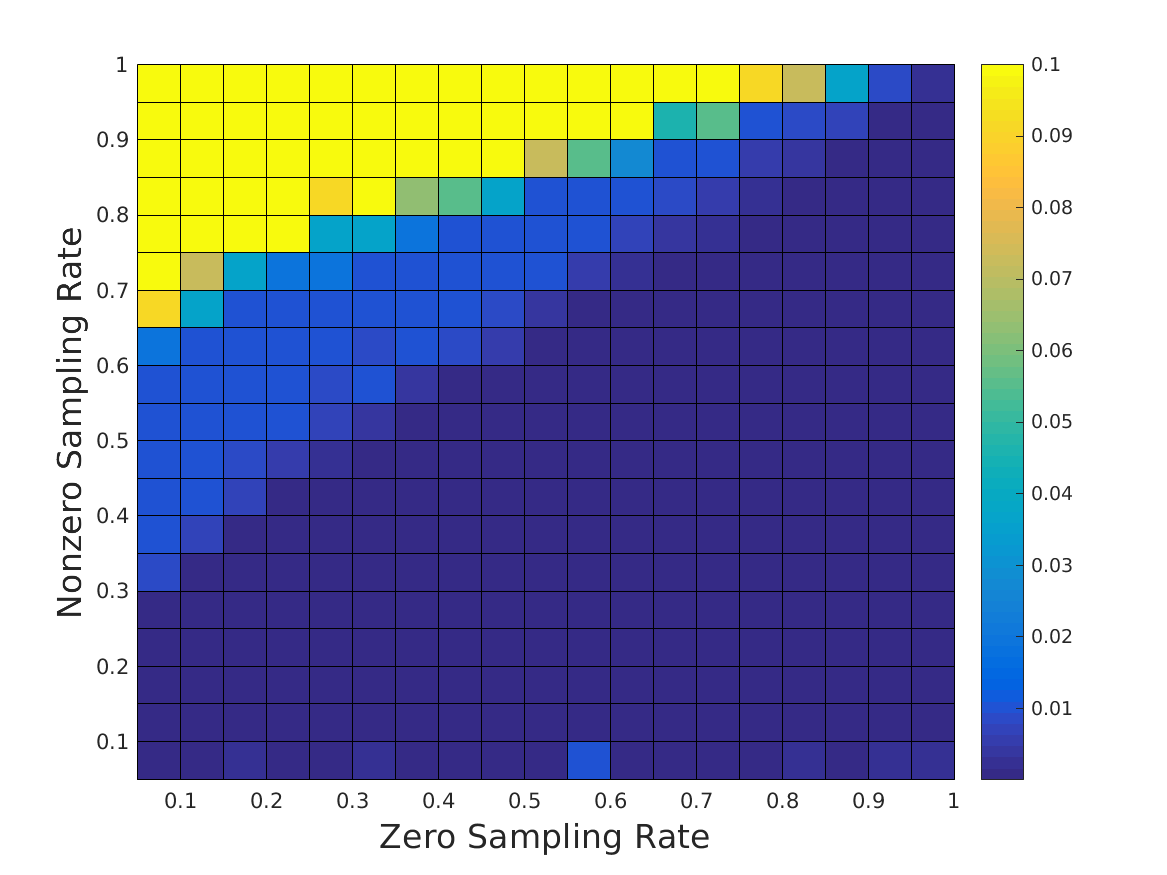}}
\caption{Average optimal $\alpha$ value among $\alpha \in \{10^{-1},10^{-2},10^{-3},10^{-4}\}$ for the minimization given in \eqref{min} with $L_1$ regularization on the recovered values for the unobserved entries in Lyme patient data.  }
\label{fig.alphaLyme}
\end{figure}

\section{Analytical Remarks}
We provide here some basic analysis of the regularization approach. 
First, in the simplified setting, in which all of the unobserved entries are exactly zero, 
the modified recovery given in \eqref{min} will always perform at least as well as traditional NNM.

\begin{prop} Suppose $M\in \mathbb{R}^{n_1\times n_2}$ and $\Omega$ gives the set of index pairs of the observed entries. Assume that all of the unobserved entries are exactly zero, i.e. $\mathcal{P}_{\Omega^C}(M) =0$. Then for 
\[\widehat M = \argmin ||A||_* \text{ s.t. } \mathcal{P}_\Omega(A) = \mathcal{P}_\Omega(M),\] and 
\[\widetilde M = \argmin ||A||_* + \alpha ||\mathcal{P}_{\Omega^C}(A)|| \text{ s.t. } \mathcal{P}_\Omega(A) = \mathcal{P}_\Omega(M), \]
we have 
\[||\widetilde M - M|| \le ||\widehat M - M||\] for any matrix norm $||\cdot ||$.
\end{prop}

\begin{proof}
From the definitions of $\widehat M$ and $\widetilde M$, 
\[||\widehat M||_* \le ||\widetilde M||_*.\]
Using the inequality above, 
\begin{align*}||\widetilde M||_*&+\alpha || \mathcal{P}_{\Omega^C}(\widetilde M)|| \le ||\widehat M||_*+\alpha||\mathcal{P}_{\Omega^C}(\widehat M)||\\
&\le ||\widetilde M||_*+\alpha||\mathcal{P}_{\Omega^C}(\widehat M)||.
\end{align*}
For $\alpha>0$, we have
\[||\mathcal{P}_{\Omega^C}(\widetilde M)|| \le ||\mathcal{P}_{\Omega^C}(\widehat M)||.\] 
The desired result then follows since 
\[\mathcal{P}_\Omega(\widetilde M) =\mathcal{P}_\Omega(\widehat M) = \mathcal{P}_\Omega(M)\]
 and under the assumption that $\mathcal{P}_{\Omega^C}(M) =0$, as
\[
||\widetilde M - M||  = || \mathcal{P}_{\Omega^C}(\widetilde M)||\le ||\mathcal{P}_{\Omega^C}(\widehat M)||=||\widehat M - M||.
\]
\end{proof}

\subsection{Connection to Robust Principal Component Analysis (RPCA)}
The program \eqref{min} very closely resembles the method proposed in \cite{RPCA}, called Robust Principal Component Analysis (RPCA). RPCA is a modified version of traditional Principal Component Analysis that is robust to rare corruptions of arbitrary magnitude. In RPCA, one assumes that a low-rank matrix has some set of its entries corrupted. The goal is to recover the true underlying matrix despite the corruptions. More simply, for the observed matrix $Y\in \mathbb{R}^{n_1\times n_2}$, we have the decomposition
\[Y = L + S,\]
where $L$ is the low-rank matrix we would like to recover and $S$ is a sparse matrix of corruptions. The strategy for finding this decomposition proposed in \cite{RPCA} is
\begin{equation}\label{RPCA_orig}
\argmin_{L,S} ||L||_* + \alpha ||S||_1\text{ s.t. } L+S=Y.
\end{equation}
This method can be extended to the matrix completion setting, in which one would like to recover unobserved values from observed values, of which a subset may be corrupted. In this setting, \cite{RPCA} proposes solving the following minimization problem
\[\argmin_{L,S}  ||L||_* + \alpha ||S||_1\text{ s.t. } \mathcal{P}_{\Omega} (L+S)=\mathcal{P}_{\Omega} (Y).\]

We now return to our original matrix completion problem, in which we assume the observed entries to be exact. Let $M\in \mathbb{R}^{n_1\times n_2}$ again be the matrix we aim to recover. If we expect the unobserved entries of $M$ to be sparse, that is, only a small fraction of them to be nonzero, we can rewrite the minimization \eqref{min} in a form similar to RPCA in which we know the support of the corruptions is restricted to the set $\Omega^C$, i.e. $S= \mathcal{P}_{\Omega^C}(S)$. We then have, 
\begin{equation}\label{RPCAmin}
\argmin_{A,S} ||A||_* + \alpha ||S||_1 \text{ s.t. } A+S=\mathcal{P}_{\Omega}(M). 
\end{equation}
 This strategy differs from traditional RPCA in that we assume the observed data to be free from errors and therefore know that the corruptions are restricted to the set of unobserved entries.

Directly applying Theorem 1.1 from \cite{RPCA}, we have the following result.
\begin{prop}\label{RPCAthm}
Suppose $M\in\mathbb{R}^{n_{1}\times n_{2}}$ and $M= U\Sigma V^*$ gives the singular value decomposition of $M$. Suppose also
\[\max_i ||U^* e_i ||^2 \le \frac{\mu r}{n_1}, \quad \max_i ||V^* e_i||^2 \le \frac{\mu r}{n_2},\]
and 
\[ ||UV^*||_\infty \le \sqrt{\frac{\mu r}{n_1 n_2}},\] where $r$ is the rank of $M$, $||X||_\infty = \max_{i,j} |X_{i,j}|$, $e_i$ is the $i^{th}$ standard basis vector and $\mu$ is the incoherence parameter as defined in \cite{RPCA}.
Suppose that the set of observed entries, $\Omega$, is uniformly distributed among all sets of cardinality of $m$ and the support set of $S_0$ of non-zero unobserved entries is uniformly distributed among all sets of cardinality $s$ contained in $\Omega^C$. Then there is a numerical constant $c$ such that with probability at least $1-cn^{-10}$ the minimization in \eqref{RPCAmin} with $\alpha = 1/\sqrt{n}$ achieves exact recovery, provided that 
\[\text{rank}(L_0)\le \rho_r n_{(2)} \mu^{-1}(\log n_{(1)})^{-2} \text{ and } s \le \rho_s n_{(1)}n_{(2)},\]
where $\rho_r$ and $\rho_s$ are positive numerical constants.
\end{prop}

This proposition is a direct application of Theorem 1.1 in \cite{RPCA} to the program given by \eqref{RPCAmin}. Note that here, the corruptions are exactly the unobserved entries that are nonzero. Thus, if $s$, the number of nonzero unobserved entries is small, this result may be stronger than corresponding matrix completion results that instead depend on $m$, the larger, number of missing entries.

The authors of \cite{RPCA} note that RPCA can be thought of as a more challenging version of matrix completion. The reasoning being, that in matrix completion we aim to recover the set of unobserved entries, whose locations are known, whereas in the RPCA setting, we have a set of corrupted entries, whose locations are unknown, and for which we would like to both identify as erroneous and determine their correct values. Figure 1 of \cite{RPCA} provides numerical evidence that in practice RPCA does in fact require more stringent conditions to achieve exact recovery than the corresponding matrix completion problem. In image completion or repair, corruptions are often spatially correlated or isolated to specific regions of an image. In \cite{repair}, the authors provide experimental evidence that incorporating an estimate of the support of the corruptions aids in recovery. By the same reasoning, we expect that a stronger result than suggested by Proposition \ref{RPCAthm} likely holds, as we do not make use of the fact that we are able to restrict the locations of the corruptions (nonzero, unobserved entries) to a subset of the larger matrix.

\section{Discussion}
For incomplete data in which we expect that unobserved entries are likely to be 0, we find that regularizing the values of the unobserved entries when performing NNM improves performance under various conditions. 
This improvement in performance holds for both synthetic data, with and without noise, as well as for Lyme disease survey data. 
We specifically investigate the performance of $L_1$ regularization on the unobserved entries as it is a natural choice for many applications.

Testing the validity of methods, such as \eqref{min}, on real data is challenging, since this setting hinges on the assumption that unobserved data is structurally different than observed data and would require having access to ground truth values for the unobserved entries. In this paper, we choose to take complete data and artificially partition it into observed and unobserved entries.
Another way to manage this challenge is to examine performance of various tasks, such as classification or prediction, based on data that has been completed in different ways. In this setting, relative performance of different completion strategies will likely depend on the specific task considered. However, for many applications, one would like to complete the data in order to use it for a further goal. In this setting, judging the performance of the matrix completion algorithm by its effect on performance of the ultimate goal is very natural.

We offer preliminary arguments as to why we might expect the approach in \eqref{min} to work well under the structural assumption that unobserved entries are likely to be sparse or small in magnitude, however, stronger theoretical results are likely possible. 
For example, we show that regularizing the values of the unobserved entries when performing NNM improves performance in the case when all unobserved entries are exactly zero, but based on empirical evidence we expect improved performance under more general conditions.

A range of papers, including \cite{candes, fazel, recht, gross}, discuss the conditions under which exact matrix completion is possible under the assumption that the observed entries of the matrix are sampled uniformly at random. Under what reasonable structural assumptions on the unobserved entries might we still be able to specify conditions that will lead to exact recovery?
We save such questions for future work.

\section*{Acknowledgments}
The authors would like to thank LymeDisease.org for the use of data derived from MyLymeData to conduct this study. We would also like to thank the patients for their contributions to MyLymeData, and Anna Ma for her guidance in working with this data. In addition, the authors were supported by NSF CAREER DMS $\#1348721$, NSF BIGDATA DMS $\#1740325$, and MSRI NSF DMS
$\#1440140$. 

\bibliographystyle{myalpha}
\bibliography{dnbib}

\end{document}